\documentclass{amsart}
\pdfoutput=1

\usepackage[utf8]{inputenc}
\usepackage{amsmath, amssymb, amsthm, verbatim, hyperref}
\usepackage{mathtools}
\usepackage[dvipsnames]{xcolor}
\usepackage{ragged2e}
\usepackage{marginnote}
\usepackage{enumerate}
\usepackage{makecell}
\usepackage{longtable}
\usepackage{epsfig}
\usepackage{tikz, graphicx} %, subfig
\usepackage{svg}
\usepackage{tikz-cd}
\usepackage[all,arc]{xy}
\usepackage{ytableau}
\usepackage{hyperref}
\usepackage{caption}
\usepackage{subcaption}
\hypersetup{hidelinks}
\usepackage[margin=1in]{geometry}

\makeatletter
\DeclareRobustCommand{\rvdots}{%
  \vbox{
    \baselineskip4\p@\lineskiplimit\z@
    \kern-\p@
    \hbox{.}\hbox{.}\hbox{.}
  }}
\makeatother

\theoremstyle{plain}
\newtheorem{thm}{Theorem}[section]
\newtheorem{lemma}[thm]{Lemma}
\newtheorem{prop}[thm]{Proposition}
\newtheorem{cor}[thm]{Corollary}

\theoremstyle{definition}
\newtheorem{definition}{Definition}[section]
\newtheorem{example}{Example}[section]

\theoremstyle{remark}

\newtheorem{question}{Question}

\usepackage{dsfont}
\usepackage[shortlabels]{enumitem}

\newcommand{\Z}{\mathbb{Z}}

\newcommand{\floor}[1]{\left\lfloor #1 \right\rfloor}

\newcommand{\Av}{{\operatorname{Av}}}

\newcommand{\SP}{{\operatorname{SP}}}

\newcommand{\Set}{{\operatorname{Set}}}
\newcommand{\phiaba}{\phi_{aba}}
\newcommand{\phisig}{\phi_{\sigma}}

\title[Deterministic stack-sorting for set partitions]{Deterministic stack-sorting for set partitions}

\author[Xia]{Janabel Xia}
\address{Department of Mathematics, Massachusetts Institute of Technology, Cambridge, MA}
\email{janabel@mit.edu}

\date{\today}

\begin{document}
\ytableausetup{centertableaux}
\allowdisplaybreaks

\begin{abstract}
    A \emph{sock sequence} is a sequence of elements, which we will refer to as \emph{socks}, from a finite alphabet. A sock sequence is \emph{sorted} if all occurrences of a sock appear consecutively. We define equivalence classes of sock sequences called \emph{sock patterns}, which are in bijection with set partitions. The notion of stack-sorting for set partitions was originally introduced by Defant and Kravitz. In this paper, we define a new deterministic stack-sorting map $\phi_{\sigma}$ for sock sequences that uses a $\sigma$-avoiding stack, where pattern containment need not be consecutive. When $\sigma = aba$, we show that our stack-sorting map sorts any sock sequence with $n$ distinct socks in at most $n$ iterations, and that this bound is tight for $n \geq 3$. We obtain a fine-grained enumeration of the number of sock patterns of length $n$ on $r$ distinct socks that are $1$-stack-sortable under $\phi_{aba}$, and we also obtain asymptotics for the number of sock patterns of length $n$ that are $1$-stack-sortable under $\phi_{aba}$. Finally, we show that for all unsorted sock patterns $\sigma \neq a\cdots a b a \cdots a$, the map $\phi_{\sigma}$ cannot eventually sort all sock sequences on any multiset $M$ unless every sock sequence on $M$ is already sorted.
\end{abstract}

\maketitle

%--------------------------------------------------

\section{Introduction}\label{sec:intro}

\subsection{Stack-sorting algorithms}\label{subsec:stack-sorting}

There is a long history of research on sorting with restricted data structures. In particular, Knuth first introduced the concept of a \emph{stack-sorting machine} in \emph{The Art of Computer Programming} \cite{knuth_art_of_programming}, which uses a stack data structure to sort sequences. Given an input sequence, one can push elements from the input onto the stack and pop elements off the stack to the output subject to the constraint that the element being popped out of the stack must be the element that was most recently pushed onto the stack. Thus, the stack grants sorting power through choices between pushes and pops. A \emph{stack-sorting algorithm} is a rule for deciding whether to push or pop at each step in the sorting process. There has since been much work on various stack-sorting algorithms inspired by Knuth's stack-sorting machine, including West's deterministic stack-sorting algorithm \cite{bona_stack_sorting_survey, defant_phd_stack_sorting_and_beyond, defant_troupes, west_permutations_1990}, pop-stack-sorting \cite{atkinson_pop_stacks_parallel, avis_pop_stacks_series, smith_stack_pop_stack_series}, deterministic stack-sorting for words \cite{defant_stack_sorting_for_words}, sorting with pattern-avoiding stacks \cite{berlow_restricted_stacks, cerbai_pattern_avoiding_machines, cerbai_restricted_stacks, defant_zheng_pattern_avoiding_stacks}, stack-sorting Coxeter groups \cite{defant_pop_stack_coxeter, defant_stack_sorting_coxeter}, and more. 

In the setting where our sequences are permutations, we say that a sequence is \emph{sorted} when it is the identity permutation. Given a stack-sorting algorithm, it is natural to ask how many passes through the stack are needed to sort the input object. We say that our input object is \emph{$k$-stack-sortable} if there is some stack-sorting algorithm that sorts the input in at most $k$ iterations. We say that our input object is \emph{$k$-stack-sortable under $\phi$} if it can be sorted by at most $k$ iterations of a given stack-sorting algorithm $\phi$. Knuth \cite{knuth_art_of_programming} characterized the permutations that are $1$-stack-sortable as exactly the permutations that avoid the pattern $231$. This characterization launched the growing field of permutation patterns research that exists today. In general, the set of permutations that are $k$-stack-sortable is closed under pattern containment. However, given the nondeterminism of Knuth's stack-sorting machine, it is very difficult to characterize or enumerate such permutations when $k \geq 2$. For example, Pierrot and Rossin only recently proved that there is a polynomial time algorithm for deciding whether or not a given permutation is $2$-stack-sortable \cite{pierrot_2_stack_sorting_poly}. This motivated West to define a deterministic stack-sorting algorithm \cite{west_permutations_1990}, which uses $21$-avoiding stacks. West's stack-sorting map is easier to analyze. For example, all permutations of length $n$ are $(n-1)$-stack-sortable under West's stack-sorting map. Furthermore, we can explicitly characterize all permutations that require all $(n-1)$ iterations of West's stack-sorting map to become sorted.

Beyond the permutation setting, Defant and Kravitz recently generalized the notion of stack-sorting to set partitions \cite{defant_foot_sorting_2022}. In this paper, we continue the investigation of stack-sorting of set partitions. More generally, we define \emph{sock sequence} to be a sequence of elements, which we refer to as \emph{socks}, from a finite alphabet. We say a sock sequence is \emph{sorted} if all occurrences of a sock appear consecutively in the sequence. We also define equivalence classes on sock sequences called \emph{sock patterns}, which are in bijection with set partitions. We will study stack-sorting on sock sequences and sock patterns.

As in the permutation setting, one can characterize which sock patterns are $1$-stack-sortable. Defant and Kravitz showed that a sock pattern $p$ is $1$-stack-sortable if and only if there is a $231$-avoiding word representation of $p$ \cite{defant_foot_sorting_2022}. However, the nondeterminism of stack-sorting sock patterns makes analyzing the number of stacks needed to sort them difficult. This motivates us to look at deterministic stack-sorting algorithms, which we believe may shed light on the nondeterministic setting. Inspired by both West's $21$-avoiding stack-sorting map and Defant and Zheng's work on consecutive pattern-avoiding stacks \cite{defant_zheng_pattern_avoiding_stacks}, we define and study novel pattern-avoiding stack-sorting maps on sock sequences. Our notion of pattern avoidance is identical to Klazar's notion of pattern avoidance for set partitions \cite{klazar-patterns-i, klazar-patterns-ii}. However, our notion of pattern avoidance within the stack differs in that we do not only consider consecutive pattern avoidance. We formally define these pattern-avoiding stack-sorting algorithms in Section~\ref{sec:prelims}.

\subsection{Main results}\label{subsec:main-results}

We give special attention to one such pattern-avoiding stack-sorting map for sock sequences. We define a \emph{foot-sorting map}, denoted $\phiaba$, based on $aba$ pattern avoidance in the stack. The name foot-sorting comes from Defant and Kravitz's original paper on the topic \cite{defant_foot_sorting_2022}, in which feet are used in place of stacks. We show that every sock sequence eventually becomes sorted under a finite number of iterations of $\phiaba$. Moreover, any sorted sock sequence stays sorted after applying $\phiaba$. Thus our map $\phiaba$ is a noninvertible sorting operator. 

When looking at any noninvertible combinatorial dynamical system, it is natural to ask how many objects require a given number of iterations to reach a periodic point. For sorting operators in particular, it is natural to ask how many objects become sorted after applying a fixed number of sorting operations. For example, Defant enumerated the $3$-stack-sortable permutations \cite{defant_stack_sorting_coxeter, defant_3_stack_sortable, defant_phd_stack_sorting_and_beyond}, West enumerated the $(n-2)$- and $(n-3)$-stack-sortable permutations \cite{west_permutations_1990}, and Claesson, Dukes, and Steingrimsson enumerated the $(n-4)$-stack-sortable permutations \cite{claesson_nm4_stack_sortable}. 

Thus, we begin by taking an enumerative approach to studying $\phiaba$. We state and prove an enumeration of sock patterns of length $n$ that are $1$-stack-sortable under $\phiaba$, and we further extend this result to a refined enumeration that is parameterized by the number of distinct socks $r$ in the sock pattern. In particular, we have the following closed forms.

\begin{thm}\label{thm:aba-sortable-coarse-enum}
    Let $s(n)$ denote the number of sock patterns of length $n$ that are $1$-stack-sortable under $\phiaba$. Let $P(x) := \sum_{n=1}^{\infty} s(n)x^n$ be the corresponding generating function. Then we have
    \[
    P(x) = \frac{(-1+3x-3x^2) + \sqrt{1-6x+7x^2-2x^3+x^4}}{4(x^2-x)}.
    \]
\end{thm}

\begin{thm}\label{thm:fine-grained-enum}
Let $s(n,r)$ denote the number of sock patterns of length $n$ and containing $r$ distinct socks that are $1$-stack-sortable under $\phiaba$. Let $P_{[q]}(x) := \sum_{n = 1}^{\infty} \sum_{r = 1}^{\infty} s(n,r)x^n q^r$ be the corresponding multivariate generating function. Then we have
    \[
    P_{[q]}(x) = \frac{(-q+(q^2+2q)x-(q^2+2q)x^2) + q\sqrt{1-2(q+2)x+(q^2+2q+4)x^2-2q^2x^3 + q^2x^4}}{2(q+1)(x^2-x)}.
    \]   
\end{thm}

We then take a separate approach to studying these pattern-avoiding stack-sorting operators, namely through investigating periodic points for general patterns $\sigma$. In particular, we show in Proposition~\ref{prop:unsortable-for-almost-all-sigma} that for any unsorted sock pattern $\sigma \neq a\cdots aba\cdots a$ and for any multiset of socks $M$ such that not all sock sequences on $M$ are already sorted, there always exists a sock sequence on $M$ that never gets sorted by $\phisig$. 

%--------------------------------------------------

\subsection{Outline of the paper}\label{subsec:outline}

In Section~\ref{sec:prelims}, we provide background discussion on sock sequences and stack-sorting algorithms that we use throughout the paper. In Section~\ref{sec:foot-sorting-map}, we describe the action of our foot-sorting map $\phiaba$ and prove a tight upper bound on the number of iterations of $\phiaba$ required to sort an arbitrary sock sequence. In Section~\ref{sec:enum}, we prove the explicit enumerations of sock patterns that are $1$-stack-sortable under $\phiaba$ stated in Theorem~\ref{thm:aba-sortable-coarse-enum} and Theorem~\ref{thm:fine-grained-enum}. In Section~\ref{sec:sigma-av-stacks}, we study the behavior of general $\sigma$-avoiding stack-sorting maps through periodic points.  In Section~\ref{sec:further-directions}, we raise some natural future directions for the investigation of deterministic stack-sorting for set partitions. 

%--------------------------------------------------

\section{Preliminaries}\label{sec:prelims}

In this section, we provide definitions and notation for the discussion of stack-sorting sock sequences.

\subsection{General notation for sock sequences}

Fix an arbitrary infinite alphabet $A$. Throughout, we will use the standard Latin alphabet $a, b, c, \dots$ to refer to distinct elements of $A$.  We call these elements of $A$ \emph{socks}. A \emph{sock sequence} is a sequence $p = p_1 \cdots p_n$ of socks in $A$. Let $A^*$ denote the (infinite) set of sock sequences on $A$.

Let $p$ be a sock sequence. Let $|p|$ denote the length of the sock sequence. We say that a sock sequence $p$ is \emph{empty} if $|p| = 0$, and \emph{nonempty} otherwise. Let $x_1x_2 \cdots$ denote the concatenation of sock sequences $x_1, x_2, \dots $ (where $x_i$ is potentially empty). Let $x^m$ denote the sock sequence $x \cdots x$ consisting of $m$ occurrences of sock $x \in A$. Let $C(p) := |\SP(p)|$ denote the number of distinct socks appearing in sock sequence $p$. 

\subsection{Sock patterns as sequences}

We first introduce the following definition, which will help us develop the notion of sock patterns.

\begin{definition}
    Let $p = p_1 \cdots p_n$ be a sock sequence on the alphabet $A$. Then define $\Set(p) := \{p_1, \dots, p_n\}$ as the set of distinct socks in $p$. 
\end{definition}

We can now define an equivalence relation on sock sequences, where two sock sequences $p = p_1 \cdots p_n$ and $q = q_1 \cdots q_m$ are equivalent if and only if there exists a bijection $f: \Set(p) \rightarrow \Set(q)$ such that $q = f(p_1) \cdots f(p_n)$. Thus, two sock sequences are equivalent if one can be obtained from the other by renaming the socks. For example, the sock sequences $abaa$ and $cacc$ are equivalent. A \emph{sock pattern} is an equivalence class on sock sequences under our equivalence relation. This also gives us a notion of pattern avoidance. We say that a sock sequence \emph{avoids} a sock pattern $\sigma$ if no subsequence of socks forms a sock sequence in the class $\sigma$.

We can also introduce the following notion of standardization to obtain a natural representative of each equivalence class.

\begin{definition}\label{def:standardization}
    The \emph{standardization} of a sock sequence $p$ is an injective renaming of the socks in $p$ such that the socks appearing for the first time from left to right form the alphabetical sequence $\{a,b,c,\dots\}$.
\end{definition}

\begin{example}\label{ex:standardization}
    Let $p = bbadb$. Then the standardization of $p$ is $aabca$.
\end{example}

We say that a sock sequence is \textit{standardized} if it is equal to its own standardization. We see then that two sock sequences are equivalent exactly when they have the same standardization. When referring to a sock pattern, we will use the unique standardized sock sequence in the equivalence class as our representative.

In the traditional setting of stack-sorting permutations, one can think of permutations as sequences on the base set $[n]$ for some $n \in \Z_{>0}$. In the setting of sock sequences, we instead define our base set to be a multiset $M$ of socks. We can then think of sock sequences as sequences on the multiset $M$. 

\begin{definition}\label{def:sock-orders-on-multiset}
    Let $M$ be a multiset of socks. Let $S(M)$ be the set of sock sequences consisting of all socks in $M$ (equivalently, the set of sock sequences on $M$). 
\end{definition}

\begin{example}\label{ex:sock-orders-on-multiset}
    Let $M = \{a,a,b,b\}$. Then $S(M) = \{aabb, abab, abba, baab, baba, bbaa\}$.
\end{example}

Note that it will often be natural to consider the entire set $S(M)$ at once, as $S(M)$ is closed under any stack-sorting map. For example, to study a stack-sorting map $\phi$, we can ask which sock sequences within $S(M)$ get sent to each other, which sock sequences within $S(M)$ eventually become sorted under repeated iterations of $\phi$, which sock sequences within $S(M)$ are periodic points, etc. 

Finally, the following notation will help us connect sock sequences to set partitions.

\begin{definition}
    Let $p$ be any sock sequence on $M$. For any sock $m \in M$, define $I(p,m) := \{i \mid p_i = m\}$, i.e.\ the set of indices $i$ at which sock $m$ occurs in $p$.
\end{definition}

\subsection{Sock patterns as set partitions}

Sock patterns of length $n$ are in bijection with set partitions on $[n]$. For example, the sock pattern represented by the standardized sock sequence $p = aabacb$ corresponds to the set partition $\{\{1,2,4\},\{3,6\},\{5\}\}$. Let $\mathcal{X_n}$ denote the set of all set partitions on $[n]$. Then the bijection is 
\begin{align*}
    \SP(p): A^* &\rightarrow \mathcal{X} \\
    p &\mapsto \{I(p,m) \mid m \in p\},
\end{align*} 
where the subsets in the set partition $\SP(p)$ correspond to the indices at which a fixed sock occurs in $p$. Note that Klazar's notion of pattern avoidance for set partitions \cite{klazar-patterns-i, klazar-patterns-ii} corresponds to our notion of pattern avoidance for sock patterns. 

%--------------------------------------------------

\subsection{A novel stack-sorting map}

We define a deterministic stack-sorting algorithm for sock sequences based on pattern avoidance of sock patterns. This sorting algorithm is inspired by the consecutive pattern-avoiding stacks defined by Defant and Zheng \cite{defant_zheng_pattern_avoiding_stacks}. Here, we consider pattern containment in a more general sense, where a pattern need not occur consecutively within the stack.

\begin{definition}\label{def:phi_aba}
    Let $\phiaba: A^* \rightarrow A^*$ denote the stack-sorting map given by the following procedure. Suppose we are given an input sequence $p$ read from left to right. Throughout the procedure, we consider the sequence $s$ of socks obtained from reading the stack from top to bottom. At each point in time, we do one of the following operations:
    \begin{enumerate}
        \item If there are remaining socks to push and pushing the leftmost sock from our input sock sequence onto the top of the stack does not create an $aba$ pattern in $s$, then we do so.
        \item Otherwise, we pop the topmost sock off of our stack.
    \end{enumerate}
    The procedure continues until the output contains all socks in the original input. Let $\phiaba(p)$ denote the output sock sequence. 
\end{definition}

\begin{example}
    Let $p = abcab$. Then the following diagram illustrates the state of the stack at each step in the algorithm. Our output sock sequence is $\phiaba(abcab) = cbbaa$, which we note is sorted.

    \begin{figure}[h!]
        \centering
        \includegraphics[width=160mm]{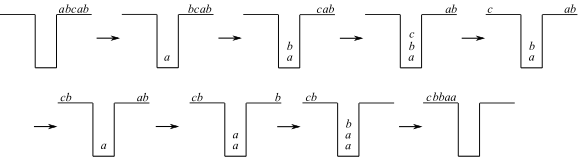}
        \caption{Result of applying map $\phiaba$ to the sock sequence $p = abcab$}
        \label{fig:enter-label}
    \end{figure}
    
\end{example}

 We can think about output sock sequences as belonging to the image of map $\phiaba$. Note that this sorting procedure preserves the multiset of socks, and therefore also induces a map $\phiaba:~S(M) \rightarrow S(M)$ for any multiset of socks $M$. Furthermore, sorting with $\phiaba$ and then renaming socks with a bijection $f$ yields the same sock sequence as renaming socks with a bijection $f$ and then sorting with $\phiaba$, so this sorting procedure also induces a map from sock patterns to sock patterns. We call this particular stack-sorting map a \emph{foot-sorting} map, which we investigate in more detail in Section~\ref{sec:foot-sorting-map}.

More generally, we can define analogous stack-sorting maps $\phisig$ for any pattern $\sigma$.

\begin{definition}\label{def:phi_sigma}
    Let $\phisig: A^* \rightarrow A^*$ denote the stack-sorting map given by the following procedure. Suppose we are given an input sequence $p$ read from left to right. Throughout the procedure, we consider the sequence $s$ of socks obtained from reading the stack from top to bottom. At each point in time, we do one of the following operations:
    \begin{enumerate}
        \item If there are remaining socks to push and pushing the leftmost sock in our input sequence onto the top of the stack does not create a $\sigma$ pattern in $s$, then we do so.
        \item Otherwise, we pop the topmost sock off of our stack.
    \end{enumerate}
    The procedure continues until the output contains all socks in the original input. Let $\phisig(p)$ denote the output sock sequence. 
\end{definition}

Again, this general sorting procedure preserves the sets of sock sequences $S(M)$ for any multiset of socks $M$ and also induces a map from sock patterns to sock patterns.

\section{The foot-sorting map $\phiaba$}\label{sec:foot-sorting-map}

In this section, we study the behavior of the foot-sorting map $\phiaba$ introduced in Definition~\ref{def:phi_aba}. We begin by presenting a quick but useful lemma that tells us what the action of our foot-sorting map $\phiaba$ looks like recursively. 

\begin{lemma}\label{phi-aba-action}
    Let $p = x^{l_1} s_1 x^{l_2} s_2 \cdots x^{l_m} s_m x^{l_{m+1}}$ be a sock sequence where $l_1, \dots, l_{m} > 0$ and $l_{m+1} \geq 0$ are integers, and $s_1, \dots, s_m$ are nonempty sock sequences not containing $x \in A$. Then
    \[
    \phiaba(p) = \phiaba(s_1)\phiaba(s_2) \cdots \phiaba(s_m)x^{l_1 + \cdots + l_{m+1}}.
    \]
\end{lemma}

\begin{proof}
    We show directly what happens to input sock sequence $p = x^{l_1} s_1 x^{l_2} s_2 \cdots x^{l_m} s_m x^{l_{m+1}}$. By definition, we first push the first $l_1$ occurrences of $x$ onto the stack. We then have the following for all $i \in [m-1]$. We i) push all socks in $s_i$ onto the stack, which only contains $l_1 + \cdots + l_i$ occurrences of $x$. This results in $s_i$ getting sorted by $\phiaba$ independently, since $s_i$ contains no occurrences of $x$ and therefore does not interact with the occurrences of $x$ at the bottom of the stack. We then ii) must pop all remaining socks in $s_i$ in the stack before pushing the next occurrences of $x$, to avoid having the pattern $aba$ occur in the stack. Finally, we iii) push the next $l_{i+1}$ consecutive occurrences of $x$ in our input. After $m-1$ rounds of this process, we are left with $s_m x^{l_{m+1}}$ in our input, $\phiaba(s_1)\cdots \phiaba(s_{m-1})$ in our output, and exactly $l_1 + \cdots + l_m$ occurrences of $x$ on our stack. Again, we push all socks in $s_m$ and independently sort them. If $l_{m+1} = 0$, we finally pop all remaining socks in $s_m$ in the stack and all $l_1 + \cdots + l_m$ occurrences of $x$ at the bottom of the stack, giving us $\phiaba(s_1) \cdots \phiaba(s_m) x^{l_1 + \cdots + l_m}$ as our final output. If $l_{m+1} > 0$, we pop all remaining socks in $s_m$ in the stack before pushing the next $l_{m+1}$ occurrences of $x$, and finally pop all $l_1 + \cdots + l_{m+1}$ occurrences of $x$ in the stack to obtain $\phiaba(s_1)\phiaba(s_2) \cdots \phiaba(s_m) x^{l_1 + \cdots l_{m+1}}$. Thus our output always becomes $\phiaba(s_1)\phiaba(s_2) \cdots \phiaba(s_m) x^{l_1 + \cdots + l_{m+1}}$, as desired.
\end{proof}

Given how our foot-sorting map $\phiaba$ acts on sock sequences, we can provide an upper bound on the number of iterations of $\phiaba$ required to sort any sock sequence. In the permutation setting, we have that any permutation of length $n$ is $(n-1)$-sortable under West's stack-sorting map, which is linear in the number of distinct elements of the permutation sequence, and furthermore this bound is tight. It turns out that any sock sequence on $n$ distinct socks is $n$-sortable under $\phiaba$, which is also linear in the number of distinct elements appearing in our sequence. We also show that this bound is tight.

\begin{thm}\label{n-$1$-stack-sortable}
    Let $p$ be any sock sequence on multiset $M$. Let $n$ be the number of distinct socks in $M$. Then $p$ is $n$-sortable under $\phiaba$. Furthermore, for any $n > 2$ there exists a sock sequence $p$ such that $p$ is not $(n-1)$-sortable under $\phiaba$.
\end{thm}

\begin{proof}
Say that a sock is \emph{clumped} if all occurrences of that sock appear consecutively. Then we can see that any sock that is clumped stays clumped under $\phiaba$; if we can push one of the clumped socks onto the stack and avoid the pattern $aba$, then we can push all of them onto the stack. Likewise, once we must pop a clumped sock out of the stack, then all of the socks can constitute the last sock in the pattern $aba$, so they must all be popped out together. 

We claim that each application of $\phi$ adds at least one more clumped sock. In particular, if the first sock $x$ is not clumped, then by Lemma~\ref{phi-aba-action}, one application of $\phiaba$ clumps together all occurrences of sock $x$ and sends them to the back of the output sock sequence. If the first sock $x$ is already clumped, then we have $p = x \cdots x q$ where $q$ is some sock sequence. Then $\phiaba(p) = \phi(q) x \cdots x$, and we can use the same argument on $q$. Thus, the first sock that is not already clumped becomes clumped under $\phiaba$. Because socks that are clumped stay clumped in any subsequent application of $\phiaba$, we need at most $n$ stacks to sort $p$.

To show that this bound is tight for any $n$, consider the sock sequence $p^* = a_1 \cdots a_n a_1 \cdots a_n$ where $a_i$ is a sock for all $i \in [n]$ and $p^*$ contains exactly two copies of each distinct sock $a_i$. We will show that $p^*$ is not $(n-1)$-sortable under $\phiaba$. To do so, we claim that the after $k$ iterations of $\phiaba$ for, our sock sequence becomes $X_1 q X_2 q X_3$ where $q$ is a sock sequence of length $|q| = n-k$ containing only single occurrences of socks, and $X_1, X_2$, and $X_3$ are sock sequences of lengths $|X_1| = 2\floor{\frac{k}{3}}, |X_2| = 2\floor{\frac{k+1}{3}}$ and $|X_3| = 2\floor{\frac{k+2}{3}}$. We prove this by induction on $k$.

Our base case is easily verifiable; $\phiaba((a_1 \cdots a_n)^2) = a_n \cdots a_2 a_n \cdots a_2 a_1 a_1)$, which we can rewrite in the form $X_1 q X_2 q X_3$ where $X_1$ and $X_2$ are empty, $X_3 = a_1a_1$, and $q = a_n \cdots a_2$, which are all of the correct form. Now we assume our hypothesis holds up to iteration $k$, and show that it remains true after iteration $k+1$. Indeed, if we write $q = aq'$, we have
\[
\phiaba(X_1 aq' X_2 aq' X_3) = (\overline{X_2} \ \overline{q'} \ \overline{X_3} \ \overline{q'} \ aa \ \overline{X_1}),
\]
which we can rewrite as $X_1' q' X_2' q' X_3'$, where $X_1' = \overline{X_2}$, $X_2' = \overline{X_3}$, and $X_3' = aa \overline{X_1}$. Then we can verify that $|X_1'| = |X_2| = 2\floor{\frac{k+1}{2}}$, $|X_2'| = |X_3| = 2\floor{\frac{k+2}{2}}$, $|X_3'| = 2 + 2\floor{\frac{k}{3}} = 2\floor{\frac{k+3}{3}}$, and $q = n-k-1$, as desired.

Given the claim, we now have that $X_2$ is nonempty for all sufficiently large $k$. The $X_1 q X_2 q X_3$ cannot be sorted until $q$ is empty, which takes at least $n$ iterations as desired.
\end{proof}

Note that this notion of nonconsecutive pattern avoidance for our deterministic stack-sorting maps $\phiaba$ (and more generally $\phisig$) is necessary for our maps to have nice properties like the above. Suppose, instead we defined an analogous stack-sorting map $\phi'_{aba}$ (and more generally $\phi'_{\sigma})$ that only considers consecutive pattern avoidance in the procedure given by Definition~\ref{def:phi_sigma}. Then one example of a sock sequence that never gets sorted by applying iterations of $\phi'_{aba}$ would be $abcabc$ (in fact, $\phi'_{aba}(abcabc) = cbacba$ and $\phi'_{aba}(cbacba) = abcabc$). Moreover, for any $\sigma$, we have sock sequences that never get sorted by applying iterations of $\phi'_{\sigma}$. Indeed, we will see in Proposition~\ref{prop:unsortable-for-almost-all-sigma} that for any pattern $\sigma \neq aba$, there always exists an unsorted sock sequence $p$ that avoids $\sigma$ and therefore alternates between $p$ and $\overline{p}$ under iterations of $\phi'_{\sigma}$, neither of which are sorted. Therefore, under the consecutive pattern avoidance regime, there would be no stack-sorting map $\phisig$ that eventually sorts any sock sequence.

\section{Enumerating $1$-stack-sortable sock patterns under $\phiaba$}\label{sec:enum}

In this section, we prove our main enumerative results. We first provide a proof of Theorem~\ref{thm:aba-sortable-coarse-enum} and then derive asymptotics based on the closed form. Then we generalize our argument to prove Theorem~\ref{thm:fine-grained-enum}. 

We first discuss the notion of appending two sock patterns. Observe that under our notion of equivalence between sock sequences, even if sock sequences $s_1$ and $s_1'$ are equivalent and sock sequences $s_2$ and $s_2'$ are equivalent, it does not necessarily follow that the sock sequence $s_1s_2$ is equivalent to the sock sequence $s_1's_2'$. For example, suppose $s_1 = aaba, s_1' = aaba, s_2 = bcab$, and $s_2' = cbac$. Then we can have $s_1s_2 = aababcab$ and $s_1's_2' = aabacbac$, which are not the same sock pattern. The sock pattern $p$ is determined by how we explicitly assign sock names to the occurrence of $s_1$ and the occurrence of $s_2$, up to standardization. 

We are now ready to prove Theorem~\ref{thm:aba-sortable-coarse-enum}.

\begin{proof}[Proof of Theorem~\ref{thm:aba-sortable-coarse-enum}]    

    We use Lemma~\ref{phi-aba-action}.  Let $p = x^{l_1} s_1 x^{l_2} s_2 \cdots x^{l_m} s_m x^{l_{m+1}}$ be a $1$-stack-sortable sock pattern under $\phiaba$ where $x \notin s_i$ for any $i \in [m]$. Note that in order for
    \[\phiaba(p) = \phiaba(s_1)\phiaba(s_2) \cdots \phiaba(s_m)x^{l_1 + \cdots + l_{m+1}}\]
    to be sorted, we must have that for all $i \in [m]$, each $s_i$ individually is a $1$-stack-sortable sock pattern under $\phiaba$. 
    
    Then given any fixed $s_i$, we only need to consider their relationship to each other in the full sock pattern. In particular, we claim that for any $l_1, l_2 > 0$ and nonempty sock patterns $s_1, s_2$ that are each $1$-stack-sortable under $\phiaba$, there are exactly two possible sock patterns $x^{l_1} s_1 x^{l_2} s_2$ that are $1$-stack-sortable under $\phiaba$. To see this, we must have $|C(s_i) \cap C(s_{i+1})| \leq 1$, as otherwise the subsequence $\phiaba(s_i) \phiaba(s_{i+1})$ in $\phiaba(p)$ contains either $abab$ or $abba$ as a subsequence and is not sorted. We then consider the following two cases:

    \begin{enumerate}[i)]
        \item If $|C(s_1) \cap C(s_{2})| = 0$, then there is exactly one way to assign sock names to $s_1$ and $s_2$ up to standardization. Namely, we can assign any sock names such that the socks in $s_1$ and $s_2$ are disjoint, and then standardize.

        \item If $|C(s_1) \cap C(s_{2})| = 1$, then again there is exactly one way to assign sock names to $s_1$ and $s_2$ up to standardization. Namely, we assign the same name to the last element that appears in $\phiaba$ and to the first element that appears in $\phiaba(s_2)$, and consider all other socks in $s_1$ and $s_2$ disjoint. 
    \end{enumerate}

    Then we can see that up to equivalence, $p$ is determined by the sock patterns $s_i$ together with the choice for $1 \leq i \leq m-1$ of whether or not $s_i$ and $s_{i+1}$ share a sock, for a total of $2^{m-1}$ choices. We can now consider the cases where $m=0$ and $m \geq 1$ separately, which allows us to write the following relation on $P(x)$:

    \begin{align}
        P(x) &= \frac{x}{1-x} + \sum_{m \geq 1} \left(P(x)^m 2^{m-1} \left(\sum_{\substack{l_1, \dots, l_m > 0 \\ l_{m+1} \geq 0}} x^{l_1 + \cdots + l_{m+1}}\right)\right) \label{eq:eqn-in-p}\\
        &= \frac{x}{1-x} + \sum_{m \geq 1} \left(P(x)^m 2^{m-1} \left(\frac{x}{1-x}\right)^m \left(\frac{1}{1-x}\right)\right) \\
        &= \frac{x}{1-x} + \frac{1}{2(1-x)}\sum_{m \geq 1} \left(P(x)^m 2^{m} \left(\frac{x}{1-x}\right)^m\right) \\
        &= \frac{x}{1-x} + \frac{1}{2(1-x)} \cdot
        \frac{2P(x) \left(\frac{x}{1-x}\right)}{1-2P(x)\left(\frac{x}{1-x}\right)}. \label{eq:nice-poly-in-p}
    \end{align}

    To see why Equation~\ref{eq:eqn-in-p} holds, we can consider the coefficient of $x^n$ for any $n \in \Z_{\geq 0}$. On the left-hand side, this is $s(n)$ by definition. On the right-hand side, we have two main terms. The first term $\frac{x}{1-x} = \sum_{n \geq 1} 1x^n$ counts the unique sock pattern with length $n$ that is $1$-stack-sortable under $\phiaba$ when $m=0$, namely $p = x \cdots x$. The second term contributes a sum of $2^{m-1} s(n_1) s(n_2) \cdots s(n_m)$ over all possible lengths $n_1, \dots, n_m, l_1, \dots, l_{m+1}$ with $n_i > 0$ for $i\in[m]$ and $l_i > 0$ such that $n_1 + \cdots + n_m + l_1 + \cdots + l_m = n$. This counts the number of sock patterns that are $1$-stack-sortable under $\phiaba$ when $m \geq 1$, since we can take any nonempty sock patterns $s_i$ that are independently $1$-stack-sortable under $\phiaba$ and append them in $2^{m-1}$ ways. 
    
    Given Equation~\ref{eq:nice-poly-in-p}, one can now easily solve the quadratic in $P(x)$ to obtain the closed form
    \[
    P(x) = \frac{(-1+3x-3x^2) + \sqrt{1-6x+7x^2-2x^3+x^4}}{4(x^2-x)},
    \]
    as desired.    
\end{proof}

% This next section will knock your socks off.

\begin{cor}\label{cor:1-sortable-asymptotics}
    We have $s(n) = K c^n n^{-3/2} + O(c^n n^{-5/2})$, where $K = 0.34313 \cdots$ is a constant and $c = 4.5464 \cdots$ is the inverse of the smallest positive root $x_0$ of $1-6x+7x^2-2x^3+x^4 = 0$.
\end{cor}

\begin{proof}
    For any power series $f$ and $n \geq 0$, let $[x^n]\{f(x)\}$ denote the coefficient of $x^n$ in the series $f(x)$. We wish to find asymptotics for $s(n) = [x^n] \{ P(x) \}$.
    
    Let $P^*(x) := \frac{\sqrt{1-6x+7x^2-2x^3+x^4}}{4(x^2-x)}$. Note that to obtain asymptotics for $s(n)$, it suffices to compute asymptotics for the coefficients of the power series of $P^*(x)$ at $x=0$, since we can verify that 
    \[P(x) - P^*(x) = \frac{-1+x+x^2}{4(x^2-x)} = -\frac{1}{4}\sum_{n=-1}^{\infty}x^n
    \]
    contributes only constant terms to $s(n)$.
    
    In particular, we define $Q(x) := P^*(x_0x)$, where $x_0$ is the smallest positive root of $1-6x+7x^2-2x^3+x^4 = 0$ (and thus the singularity of $P^*(x)$ with smallest magnitude). Then we have that the smallest singularity of $Q(x)$ occurs at $x = 1$, so we can write $Q(x) = (1-x)^{\beta} R(x)$ where $\beta = \frac{1}{2}$ and $R(x)$ is analytic in some disk $|x| < 1 + \varepsilon$ where $\varepsilon > 0$. Let $R(x) = \sum_{j=0}^{\infty} r_j(1-x)^j$. Then Darboux's Theorem \cite{darboux_1878, wilf_gfology} tells us that
    \[
    [x^n]\{(1-x)^{\beta}R(x)\} = [x^n]\left\{\sum_{j=0}^{m}r_j(1-x)^{\beta + j}\right\} + O(n^{-m-\beta-2}).
    \]
    Plugging in $m=0$ and $\beta = \frac{1}{2}$ gives us
    \[
    [x^n]\{(1-x)^{1/2}R(x)\} = [x^n]\{r_0(1-x)^{1/2}\} + O(n^{-5/2}).
    \]
    Note that $r_0 = R(1)$, which is well-defined and computable. We can also compute $[x^n]\{(1-x)^{1/2}\}$ using Stirling's approximation as follows: 
    \begin{align*}
        [x^n]\{(1-x)^{1/2}\} &= -(-1)^{n-1} \binom{\frac{1}{2}}{n} \\
        &= -\left(\frac{1}{2}\right)^n \frac{1 \cdot 3 \cdot \dots \cdot (2n-3)}{n!} \\
        &= -\left(\frac{1}{4}\right)^n \frac{1}{2n-3} \frac{(2n)!}{(n!)^2} \\
        &= -\left(\frac{1}{4}\right)^n \frac{1}{2n-3} \frac{\left(\frac{2n}{e}\right)^{2n} \sqrt{2\pi (2n)}}{\left(\left(\frac{n}{e}\right)^{n} \sqrt{2\pi n}\right)^2} \cdot \left(1+O(n^{-1})\right) \\
        &= -\frac{1}{\sqrt{4\pi}} n^{-3/2} \cdot \left(1+O(n^{-1})\right) \\
        &= -\frac{1}{\sqrt{4\pi}} n^{-3/2} + O(n^{-5/2}).
    \end{align*}

    Thus, we have 
    \begin{align*}
        [x^n]\{Q(x)\} &= -\frac{R(1)}{\sqrt{4\pi}} n^{-3/2} + O(n^{-5/2})\\
        \Rightarrow [x^n]\{P^*(x)\} &= -\left(\frac{1}{x_0}\right)^n \frac{R(1)}{\sqrt{4\pi}} n^{-3/2} + O\left(\left(\frac{1}{x_0}\right)^n n^{-5/2}\right) \\
        \Rightarrow s(n) &= K c^n n^{-3/2} + O(c^n n^{-5/2}),
    \end{align*}
    where $K = -\frac{R(1)}{\sqrt{4\pi}} = 0.34313 \cdots$ is a constant, $c = \frac{1}{x_0}$, and $x_0 = \frac{1}{2}\left(1-\sqrt{8\sqrt{2} - 11}\right)$ is the smallest solution to $1-6x+7x^2-2x^3+x^4 = 0$, as desired.
\end{proof}

We can extend the previous result to a more refined enumeration of the $1$-stack-sortable sock patterns under $\phiaba$ of a given length $n$ based on the number of distinct socks that appear. We can thus define a $q$-analogue for the statistic $C(p)$, and find a closed form for its generating function as stated in Theorem~\ref{thm:fine-grained-enum}. Our proof is analogous to the proof of Theorem~\ref{thm:aba-sortable-coarse-enum}. 

\begin{proof}[Proof of Theorem~\ref{thm:fine-grained-enum}]
     Let $s(n,r)$ denote the number of sock patterns of length $n$ that are $1$-stack-sortable under $\phiaba$ and have $r$ distinct socks. Again, for $p = x^{l_1} s_1 x^{l_2} s_2 \cdots x^{l_m} s_m x^{l_{m+1}}$ to be $1$-stack-sortable, we must have that $s_i$ are individually $1$-stack-sortable under $\phiaba$ for all $i \in [m]$ (where the $s_i$ are potentially nonempty). Furthermore, if $p$ has $r$ distinct socks, then we have to choose how to cover those $r$ socks across the $m$ sock patterns $s_1, \dots, s_m$.

    We again split into two cases. When $m=0$, we have that there is exactly one sock pattern of length $n$ that is $1$-stack-sortable under $\phiaba$, namely $p = x^n$. This pattern contains one distinct sock, so we can capture these sock patterns in the sum $\sum_{n > 0} x^n q^1 = \frac{qx}{1-x}$.
    
    Now suppose $m > 0$, so we have a nonzero number of nonempty sock patterns $s_i$. Suppose $s_i$ has length $n_i > 0$ and $r_i > 0$ distinct socks for $i \in [m]$. For each nonempty $s_i$, we can again either choose to i) have the first sock of $\phiaba(s_i)$ be the same as the last sock of $\phiaba(s_{i-1})$, or ii) have $\phiaba(s_i)$ (and therefore $s_i$) have disjoint socks from $\phiaba(s_{i-1})$. In the former, the total number of distinct socks added by appending $s_i$ is $r_i - 1$, and in the latter, the total number of distinct socks added by adding $s_i$ is $r_i$. Thus, our contributions to $P_{[q]}(x)$ look like the product of either $s(n_i, r_i) x^{n_i} q^{r_i}$ or $s(n_i, r_i) x^{n_i} q^{r_i-1}$ terms over all $i \in [m]$, which we can factor as $P_{[q]}(x)^m (1+q^{-1})^m$. We can thus write the following relation on $P_{[q]}(x)$:

    \begin{align*}
        P_{[q]}(x) &= \frac{qx}{1-x} + \sum_{m \geq 1} \left(P_{[q]}(x)^m (1+q^{-1})^{m-1} \left(\sum_{\substack{l_1, \dots, l_m > 0 \\ l_{m+1} \geq 0}} x^{l_1 + \cdots + l_{m+1}} q \right)\right)\\
        &= \frac{qx}{1-x} + \sum_{m \geq 1} \left(P_{[q]}(x)^m (1+q^{-1})^{m-1} \left(\frac{x}{1-x}\right)^m \left(\frac{1}{1-x}\right) q \right) \\
        &= \frac{qx}{1-x} + \frac{q}{(1+q^{-1})(1-x)} \sum_{m \geq 1} \left(P_{[q]}(x)^m (1+q^{-1})^{m} \left(\frac{x}{1-x}\right)^m\right) \\
        &= \frac{qx}{1-x} + \frac{q}{(1+q^{-1})(1-x)} \cdot \frac{P_{[q]}(x)(1+q^{-1})\left(\frac{x}{1-x}\right)}{1-P_{[q]}(x)(1+q^{-1})\left(\frac{x}{1-x}\right)}.
    \end{align*}

    We can now easily solve this quadratic in $P_{[q]}(x)$ to get the closed form
    \[
    P_{[q]}(x) = \frac{(-q+(q^2+2q)x-(q^2+2q)x^2) + q\sqrt{1-2(q+2)x+(q^2+2q+4)x^2-2q^2x^3 + q^2x^4}}{2(q+1)(x^2-x)},
    \]    
    as desired.
\end{proof}

After obtaining a closed form and asymptotics for the number of $1$-stack-sortable sock patterns of a given length, it is natural to ask whether or not there is a simple way to describe what $1$-stack-sortable sock sequences look like. It turns out that even though we know recursively what the map $\phiaba$ behaves like, we cannot easily characterize the set of $k$-stack-sortable sock sequences. The following example shows that sortability is not closed under containment, that is, there exist sock sequences $p$ and $q$ such that $p$ is contained in $q$, $p$ is not $1$-stack-sortable, and $q$ is $1$-stack-sortable. Take $p=abcabc$ and $q=babcabc$. Then $\phiaba(p) = cbcbaa$ is unsorted and $\phiaba(q) = aaccbbb$ is sorted. Therefore, the set of $1$-stack-sortable sock sequences under $\phiaba$ cannot be characterized as all sock sequences that avoid some fixed set of sock patterns. Note that this contrasts with the original nondeterministic setting that Defant and Kravitz introduced based on Knuth's stack-sorting machine \cite{defant_foot_sorting_2022}. The lack of a pattern avoidance characterization of $1$-stack-sortable sock patterns is due to the determinism of the stack-sorting procedure here. It is worth mentioning that the set of $k$-stack-sortable permutations under West's stack-sorting map is also not closed under permutation pattern containment when $k>1$ \cite{west_permutations_1990}.

\section{General $\sigma$-avoiding stack-sorting algorithms}\label{sec:sigma-av-stacks}

In this section, we study the behavior of $\phisig$ for general $\sigma$ beyond the pattern $aba$. A natural question to ask is whether or not $\phisig$ will always sort any arbitrary sock sequence in $A^*$. The following proposition tells us that $\phiaba$ is the only such map.

\begin{prop}
    The only $\sigma$ for which $\phisig$ eventually sorts all sock sequences is $\sigma = aba$.
\end{prop}

\begin{proof}
    Note that for any $\sigma \neq aba$ such that $|\sigma| \geq 3$, then $\phisig(aba) = aba$ and thus will never sort $aba$. Thus, we only need to consider when $|\sigma| \leq 2$. There are two patterns of length two, namely $ab$ and $aa$. We can easily verify that the sock sequence $abcabc$ gets taken to itself and to its reverse under $\phi_{ab}$ and $\phi_{aa}$, respectively. 
\end{proof}

The question above is not particularly interesting, since we simply show that there exist short sock sequences that cannot be sorted by $\phisig$ when $\sigma$ is sufficiently large. We would like to say something stronger. In the permutation setting, it makes sense to group permutations of a given length $n$ (since passes through the stack will never change the length of the input permutation) and say something about the existence of unsortable permutations for all $n$. Similarly, in the sock sequence setting, we group sock sequences $S(M)$ on a given multiset $M$. Now we can ask for which $(\sigma, M)$ do we have that any sock sequence within $S(M)$ will eventually become sorted under $\phisig$. We have the following proposition, which tells us that for an infinitely large class of $\sigma$, the map $\phisig$ cannot eventually sort all sock sequences in \emph{any} $S(M)$ (unless every sock sequence of $S(M)$ is already sorted).

\begin{prop}\label{prop:unsortable-for-almost-all-sigma}
    Let $\sigma$ be an unsorted sock pattern that is not of the form $a\cdots a b a \cdots a$. Then for any multiset of socks $M$ such that not all sock sequences in $S(M)$ are sorted, there exists a sock sequence $p^* \in S(M)$ that is not $k$-stack-sortable under $\phisig$ for any $k$.  
\end{prop}

\begin{proof}

It suffices to show that there exists an unsorted sock sequence $p$ such that $p \in \Av(\sigma, \overline{\sigma})$, as applying $\phisig$ to any sock sequence $p$ that avoids $\sigma$ just reverses $p$. Suppose that $M$ consists of distinct socks $a_1, a_2, \dots, a_m$. Without loss of generality, suppose there are at least two copies of $a_1$. We know that $\sigma$ contains $aba$ because it is unsorted. Thus we can consider the following cases for $\sigma$ and in each case we explicitly construct a sock sequence $p^*$ that avoids both $\sigma$ and $\overline{\sigma}$:

\textbf{Case 1:} $\sigma$ contains $abba$ or $abca$. Let $p = a_1 \cdots a_1 a_2 \cdots a_2 \cdots a_n \cdots a_n$ be a sorted sock sequence. Consider the sock sequence 
\[p^* = a_1 \cdots a_1 a_2 a_1 a_2 \cdots a_2 a_3 \cdots a_3 \cdots a_m \cdots a_m\]
obtained by swapping the last $a_1$ and the first $a_2$. Then $p^*$ avoids both $\sigma$ and $\overline{\sigma}$ because the only occurrence of $aba$ would either consist of the last two occurrences of $a_1$ or the first two occurrences of $a_2$, both of which only contain one sock between them.

\textbf{Case 2:} $\sigma$ contains $abab, abac$ or $caba$. Then we consider the sock sequence 
\[p^* = a_1 \cdots a_1 a_2 \cdots a_2 \cdots a_m \cdots a_m a_1.\]
This cannot contain an occurrence of $\sigma$ because the only opportunity to contain $aba$ as a pattern is to contain one of the occurrences of $a_1$ at the beginning and the last occurrence $a_1$. Then there would be no more socks on either end of the sock sequence to represent the sock not equal to $a_1$ in the pattern $\sigma$. For the same reason, $p^*$ cannot contain an occurrence of $\overline{\sigma}$.

Note that these cases are exhaustive, as we assume $\sigma$ is not of the form $a \cdots a b a \cdots a$. Therefore we have found a sock sequence $p^*$ that avoids $\sigma$ and $\overline{\sigma}$ for all possible $\sigma$, as desired.
\end{proof}

\section{Further Directions}\label{sec:further-directions}
In this final section, we propose several questions and directions for future study. We hope that continued investigation of deterministic stack-sorting maps for sock sequences, whether based on pattern avoidance or not, is fruitful and yields implications for the nondeterministic setting. We provide possible directions along the lines of characterization and enumeration.

Recall that we have shown that the foot-sorting map $\phiaba$ can sort any sock sequence on $n$ socks in at most $n$ iterations. Furthermore, we provided a construction $p = (a_1 \cdots a_n)^2$ to show that this bound is tight. We may ask whether there are other sock sequences for which the bound is tight. 

\begin{question}
    What are all sock sequences on $n$ socks that are not $(n-1)$-sortable under $\phiaba$?
\end{question}

It could be interesting to characterize the sock sequences on $n$ distinct socks that are not $(n-1)$-stack-sortable. Beyond $(n-1)$-sortable, it would also be interesting to characterize the sock sequences that are $k$-stack-sortable and not $(k-1)$-stack-sortable for other values of $k$. For example, while there is no pattern avoidance characterization of the $1$-stack-sortable sock sequences, it would be interesting to see whether or not there is still some nice characterization.

Along the lines of characterization, it could also be interesting to characterize an alternate class of periodic points. When looking at $\phiaba$, we know that all sock sequences eventually get sorted, and thus the only periodic points are the sorted sock sequences, which are precisely those avoiding the sock pattern $aba$. Any sock sequence avoiding $\sigma$ and $\overline{\sigma}$ is a periodic point under $\phisig$ with period two. Recall that we have also shown in Proposition~\ref{prop:unsortable-for-almost-all-sigma} that for almost all $\sigma$, we can find an unsorted periodic point avoiding $\sigma$ and $\sigma$ that has period two. It is natural to ask whether or not there are periodic points under $\phisig$ with period greater than two when $\sigma \neq aba$. It could be interesting to determine for which $\sigma$ there exist periodic points under $\phisig$ with period greater than two, and to characterize such periodic points.

Also recall that we have an enumeration of the sock patterns of length $n$ that are $1$-stack-sortable under $\phiaba$. Because we do not have a nice way to understand what the $1$-stack-sortable sock sequences look like, nor do we understand how to count the number of preimages in general, it is not immediately clear how to enumerate even $2$-stack-sortable sock patterns.

\begin{question}
    How many sock patterns of length $n$ are $k$-stack-sortable under $\phiaba$ for $k > 1$?
\end{question}
 
Another line of investigation would be to look at other deterministic sorting algorithms, such as considering analogous sorting maps $\phi_{\sigma_1, \dots, \sigma_i}$ that avoid multiple sock patterns $\sigma_1, \dots, \sigma_i$ at once. It could be interesting to investigate the behavior of alternative deterministic stack-sorting algorithms on sock sequences in general.

%--------------------------------------------------

\section{Acknowledgments}

This work was done at the University of Minnesota Duluth with support from Jane Street Capital, the National Security Agency, and the CYAN Undergraduate Mathematics Fund at MIT. The author is grateful to Joe Gallian and Colin Defant for directing the Duluth REU program and for the opportunity to participate. The author is also grateful to Mitchell Lee, Noah Kravitz, Maya Sankar, and Yelena Mandelshtam for helpful discussions and their guidance throughout the research process.

\bibliographystyle{amsplain0}
\bibliography{bibliography}

\end{document}